\documentclass[10pt,reqno,a4paper,oneside,11pt]{amsart}%
\usepackage{amsfonts}
\usepackage{amsfonts}
\usepackage{amsfonts}
\usepackage{amsfonts}
\usepackage{amsfonts}
\usepackage{amsfonts}
\usepackage{amsfonts}
\usepackage{amsfonts}
\usepackage{amsfonts}
\usepackage{amsfonts}
\usepackage{amsfonts}
\usepackage{amsfonts}
\usepackage{amsfonts}
\usepackage{amsfonts}
\usepackage{amsfonts}
\usepackage{amsfonts}
\usepackage{amsfonts}
\usepackage{amsfonts}
\usepackage{amsfonts}
\usepackage{amsfonts}
\usepackage{mathrsfs}
\usepackage{mathrsfs}
\usepackage{amsfonts}
\usepackage{amssymb}
\usepackage{amsmath}
\usepackage{amsthm}
\usepackage{graphicx}
\usepackage{color}

\providecommand{\U}[1]{\protect\rule{.1in}{.1in}}
\newtheorem{theorem}{Theorem}[section]

\newtheorem{remark}[theorem]{Remark}

\theoremstyle{definition}
\theoremstyle{remark}
\numberwithin{equation}{section}

\ifx\pdfoutput\relax\let\pdfoutput=\undefined\fi
\newcount\msipdfoutput
\ifx\pdfoutput\undefined\else
\ifcase\pdfoutput\else
\ifx\paperwidth\undefined\else
\ifdim\paperheight=0pt\relax\else\pdfpageheight\paperheight\fi
\ifdim\paperwidth=0pt\relax\else\pdfpagewidth\paperwidth\fi
\fi\fi\fi
\begin{document}
\pagestyle{myheadings}

\begin{center}
{\huge \textbf{The complete equivalence canonical form of four matrices over an arbitrary  division ring}}\footnote{This research was supported by the grants from the National Natural
Science Foundation of China (11571220) and  the
Natural Sciences and Engineering Research Council of Canada (NSERC) (RGPIN
312386-2015).
\par
{}* Corresponding author. \par  Email address: wqw@t.shu.edu.cn, wqw369@yahoo.com (Q.W. Wang); hzh19871126@126.com (Z.H. He); yang.zhang@umanitoba.ca (Y. Zhang)}

\bigskip

{ \textbf{Zhuo-Heng He$^{a,b}$, Qing-Wen Wang$^{a,*}$, Yang Zhang$^{c}$}}

{\small
\vspace{0.25cm}

$a.$ Department of Mathematics, Shanghai University, Shanghai 200444, P. R. China\\

$b.$ Department of Mathematics and Statistics, Auburn University, AL 36849-5310, USA\\

$c.$ Department of Mathematics, University of Manitoba, Winnipeg, MB R3T 2N2, Canada}

\end{center}

\vspace{1cm}

\begin{quotation}
\noindent\textbf{Abstract:} In this paper, we give the complete structures of the equivalence canonical form of four matrices over an arbitrary division ring.
As applications,   we derive some practical necessary and sufficient conditions for the solvability to some systems of generalized Sylvester matrix equations using the ranks of their coefficient matrices. The results of this paper are new and available over the real number field, the complex number field, and the quaternion algebra.\vspace{3mm}

\end{quotation}

\section{\textbf{Introduction}}
The singular value decomposition (SVD) is not only a basic approach  in matrix theory, but also a powerful technique for solving problems in such diverse
applications as signal processing, statistics, system and control theory and psychometrics (e.g., \cite{moor6}, \cite{Deprettere}, \cite{alter01}, \cite{alter03}, \cite{alter02}, \cite{alter04}). To solve various problems, people have already extended the SVD method to a set of matrices instead of a single matrix (e.g., \cite{dlch}-\cite{moorbelgium}, \cite{FERNANDO}-\cite{heath}, \cite{7777}, \cite{helaa}, \cite{hewangamc2017}, \cite{ccp}, \cite{QWWangandyushaowen}, \cite{zhangxia}, \cite{van}, \cite{tree1}).

Matrix decomposition over an arbitrary division ring is an important part of matrix theory and division ring theory. Since
the commutative law of the multiplication does not hold over an arbitrary division ring, the results of matrix decompositions over an arbitrary division ring have
not been more fruitful so far than those over fields. Wang, van der Woude, and Yu  \cite{QWWangandyushaowen} proposed an equivalence canonical form of a matrix triplet
\begin{align}\label{equ011}
\bordermatrix{
~& p  \cr
m&A\cr
s&C\cr
t&D }
\end{align}
over an arbitrary division ring with the same number of columns.

We in this paper investigate the equivalence canonical form of four matrices
\begin{align}\label{array1}
\bordermatrix{
~& p & q \cr
m&A&B\cr
s&C&~\cr
t&D&~ },
\end{align}
where $A (m\times p),
B (m\times q),C (s\times p)$ and $D (t\times p)$ are  arbitrary  general matrices over an arbitrary  division ring $\mathcal{F}$. In the matrix array (\ref{array1}) with
four matrices, the matrices $A$ and $B$ have the same number of rows, meanwhile $A,C$ and $D$ have the same number of columns. Note that the matrix triplet
(\ref{equ011}) is a special case of (\ref{array1}).

In 2012, Wang, Zhang and van der Woude \cite{zhangxia} gave an equivalence canonical form of the general matrix   quaternity (\ref{array1}) over $\mathcal{F}$. They proved how to find nonsingular matrices $M,P,Q,$ and $S$ such that
\begin{align}
MAP=S_{a},\quad MBQ=S_{b},\quad SCP=S_{c},\quad TDP=S_{d}.
\end{align}
The matrices $S_{a},S_{b},S_{c},S_{d}$ are quasi-diagonal matrices, their only nonzero blocks being identity matrices (see \cite{zhangxia}).

\begin{theorem}\label{lemma01}\cite{zhangxia} Consider the equivalence canonical form of the general matrix   quaternity (\ref{array1}), where $A (m\times p),
B (m\times q),C (s\times p)$ and $D (t\times p)$ are  arbitrary  general matrices over an arbitrary  division ring $\mathcal{F}$. Then, there exist nonsingular matrices
 $M,P,Q,S$ and $T$ such that
\begin{align}
MAP=S_{a},\quad MBQ=S_{b},\quad SCP=S_{c},\quad TDP=S_{d},
\end{align}
where
\begin{align}\label{equ022}
S_{a}=\bordermatrix{
~&r_{1}&r_{2}& \cr
~&0&I&0\cr
~&0&0&0\cr
~&I&0&0\cr
~&0&0&0
},
S_{b}=
\bordermatrix{
~&\mathbf{rank}(B)& \cr
~&I&0\cr
~&0&0
},
\end{align}
\begin{align}
S_{c}=\bordermatrix{
~&r_{5}&r_{1}-r_{5}&r_{4}&r_{2}-r_{4}&r_{3}& \cr
~&I&0&0&0&0&0\cr
~&0&0&I&0&0&0\cr
~&0&0&0&0&I&0\cr
~&0&0&0&0&0&0
},
\end{align}
\begin{align}\label{equ024}
S_{d}=(\Delta_{4}^{1},\Delta_{4}^{2}),
\end{align}
\begin{align}
\Delta_{4}^{1}=
\bordermatrix{
~&r_{11}&r_{5}-r_{11}&r_{10}&r_{13}&r_{14}&r_{\theta}&r_{13}&r_{9}-r_{13}&r_{4}-r_{9}\cr
~&I&0&0&0&0&0&0&0&0\cr
~&0&0&I&0&0&0&0&0&0\cr
~&0&0&0&I&0&0&I&0&0\cr
~&0&0&0&0&0&0&0&I&0\cr
~&0&0&0&0&0&0&0&0&0\cr
~&0&0&0&0&I&0&0&0&0\cr
~&0&0&0&0&0&0&0&0&0\cr
~&0&0&0&0&0&0&0&0&0\cr
~&0&0&0&0&0&0&0&0&0\cr
~&0&0&0&0&0&0&0&0&0
}
\end{align}
\begin{align}
\Delta_{4}^{2}=
\bordermatrix{
~&r_{8}&r_{12}&r_{\pi}&r_{12}&r_{14}&r_{7}-r_{12}-r_{14}&r_{3}-r_{7}&r_{6}&\cr
~&0&0&0&0&0&0&0&0&0\cr
~&0&0&0&0&0&0&0&0&0\cr
~&0&0&0&0&0&0&0&0&0\cr
~&0&0&0&0&0&0&0&0&0\cr
~&0&I&0&I&0&0&0&0&0\cr
~&0&0&0&0&I&0&0&0&0\cr
~&0&0&0&0&0&I&0&0&0\cr
~&I&0&0&0&0&0&0&0&0\cr
~&0&0&0&0&0&0&0&I&0\cr
~&0&0&0&0&0&0&0&0&0
},
\end{align}
\begin{align}\label{xin110}
r_{\theta}=r_{1}-r_{5}-r_{10}-r_{13}-r_{14},~r_{\pi}=r_{2}-r_{4}-r_{8}-r_{12},
\end{align}
\begin{align}
r_{1}=\mathbf{rank}(A,~B)-\mathbf{rank}(B),~r_{2}=\mathbf{rank}(A)+\mathbf{rank}(B)-\mathbf{rank}(A,~B),
\end{align}
\begin{align}
r_{3}=\mathbf{rank}\begin{pmatrix}A\\C\end{pmatrix}-\mathbf{rank}(A),~r_{4}=\mathbf{rank}\begin{pmatrix}A&B\\C&0\end{pmatrix}+\mathbf{rank}(A)-\mathbf{rank}(A,~B)
-\mathbf{rank}\begin{pmatrix}A\\C\end{pmatrix},
\end{align}
\begin{align}
r_{5}=\mathbf{rank}(C)+\mathbf{rank}(A,~B)-\mathbf{rank}\begin{pmatrix}A&B\\C&0\end{pmatrix},
\end{align}
\begin{align}
r_{6}=\mathbf{rank}\begin{pmatrix}A\\C\\D\end{pmatrix}-\mathbf{rank}\begin{pmatrix}A\\C\end{pmatrix},
~r_{7}=\mathbf{rank}\begin{pmatrix}A\\D\end{pmatrix}+\mathbf{rank}\begin{pmatrix}A\\C\end{pmatrix}-\mathbf{rank}(A)-\mathbf{rank}\begin{pmatrix}A\\C\\D\end{pmatrix},
\end{align}
\begin{align}\label{equhh028}
r_{8}+r_{9}=\mathbf{rank}\begin{pmatrix}A&B\\D&0\end{pmatrix}+\mathbf{rank}(A)-\mathbf{rank}\begin{pmatrix}A\\C\end{pmatrix}-\mathbf{rank}\begin{pmatrix}A\\D\end{pmatrix},
\end{align}
\begin{align}
r_{8}+r_{12}=\mathbf{rank}\begin{pmatrix}A&B\\C&0\\D&0\end{pmatrix}+\mathbf{rank}\begin{pmatrix}A\\C\end{pmatrix}
-\mathbf{rank}\begin{pmatrix}A&B\\C&0\end{pmatrix}-\mathbf{rank}\begin{pmatrix}A\\C\\D\end{pmatrix},
\end{align}
\begin{align}
r_{10}+r_{11}=\mathbf{rank}(A,~B)+\mathbf{rank}(D)-\mathbf{rank}\begin{pmatrix}A&B\\D&0\end{pmatrix},
\end{align}
\begin{align}\label{equhh029}
r_{10}+r_{13}+r_{14}=\mathbf{rank}\begin{pmatrix}C\\D\end{pmatrix}+\mathbf{rank}\begin{pmatrix}A&B\\C&0\end{pmatrix}-\mathbf{rank}(C)
-\mathbf{rank}\begin{pmatrix}A&B\\C&0\\D&0\end{pmatrix}.
\end{align}

\end{theorem}

The theorem mentioned above  is important to matrix theory. However, so far, the explicit expressions of  $r_{8},r_{9},r_{10},r_{11},r_{12},r_{13},$ and $r_{14}$ have not been given except for giving the relations among $r_{8},r_{9},r_{10},r_{11},r_{12},r_{13},$ and $r_{14}$ in \cite{zhangxia}.  Clearly, we only find the explicit expressions of  $r_{8},r_{9},r_{10},r_{11},r_{12},r_{13},$ and $r_{14}$, the equivalence canonical form of the general matrix quaternity (\ref{array1}) is complete. The main aim of this paper is to give the complete equivalence canonical form of the
matrix quaternity (\ref{array1}) in terms of the ranks of the matrices $A,B,C,$ and $D.$  In order to solve this challenging problem, we need
to construct some more complicated block matrices and investigate their relationships. Another motivation of this paper is that by the complete equivalence canonical form of the general matrix quaternity (\ref{array1}), we investigate some practical solvability conditions for the following two systems of generalized Sylvester-type matrix equations

\begin{align}
\label{system02}
\left\{
\begin{array}{rll}
AXE+BYH & = & \Phi,\\ CXF & = & \Psi,\\ DXG & = & \Omega
\end{array}
  \right.
\end{align}
and
\begin{align}
\label{system05}
\left\{
\begin{array}{rll}
AXE+BY+ZH & = & \Phi,\\
CXF & = & \Psi,\\
DXG & = & \Omega,
\end{array}
  \right.
\end{align}
where  $A,B,C,D,E,F,G,H,\Phi,\Psi,$ and $\Omega$ are given matrices over an arbitrary division ring. Note that the system (\ref{system02}) and (\ref{system05}) contain the classical matrix equations
\begin{align}
\label{system10}
\left\{
\begin{array}{rll}
AXE & = & \Phi \\
CXF & = & \Psi \\
DXG & = & \Omega
\end{array}
  \right.
\end{align}
as a special case. Some non-interacting control problems can be transformed into the system (\ref{system10}) (see, \cite{JW01}) which was further investigated recently in \cite{HWang}. Matrix equations have been one of the main topics in matrix theory and its
applications. There have been many papers using different approaches to investigate the matrix equations
(e.g., \cite{xibanya01}-\cite{Dmytryshyn}, \cite{duan1}, \cite{wanghe2}-\cite{mao16}, \cite{A.B}, \cite{wanghe4444444}-\cite{JW01}, \cite{dxiuxie2}-\cite{yuan2}, \cite{wangronghao}).

The remainder of the paper is organized as follows. In Section 2, we give the explicit expressions of  $r_{8},\ldots,r_{14}$ in Theorem \ref{lemma01}.  In Section 3, as applications of the main results in Section 2, we derive some necessary and sufficient conditions for the solvability to the systems (\ref{system02}) and (\ref{system05}) by using the ranks of the given matrices. The conditions presented in this paper are simpler
and practical comparing with the solvability conditions given in \cite{zhangxia}. Here we only need to check whether these rank conditions hold or not.
Moreover, some solvable conditions are given for Hermitian solutions to some systems of matrix equations over a division ring with an involution.

Throughout this paper, we denote the set of all $m\times n$ matrices over $\mathcal{F}$ by $\mathcal{F}^{m\times
n}$, the set of all $n\times n$ invertible matrices over $\mathcal{F}$ by $GL_{n}(\mathcal{F})$.   Let $ A \in \mathcal{F}^{m\times
n}$, the symbols $I,\mathcal{C}(A),\mathcal{R}(A)$ and $\mbox{dim} \mathcal{R}(A)$ stand for the
 identity matrix, the column right space of $A$, the row left space of   $A$   and the dimension of $\mathcal{R}(A),$ respectively.  It is well-known that $\mbox{dim} \mathcal{R}(A)=\mbox{dim} \mathcal{C}(A),$ which is called the rank of $A$ and denoted by $\mathbf{r}(A).$ The symbol $A^{*}$ stands for the conjugate transpose of $A$ over an arbitrary division ring with an involutive anti-automorphism ``$*$''.

\section{\textbf{The explicit expressions of  $r_{8},\ldots,r_{14}$ in Theorem \ref{lemma01}}}

In this Section, we consider the explicit expressions of  $r_{8},\ldots,r_{14}$ in Theorem \ref{lemma01}, which is a challenging problem in \cite{zhangxia}.

\begin{theorem}\label{theorem01}
Consider the equivalence canonical form of the general matrix   quaternity (\ref{array1}), where $A\in \mathcal{F}^{m\times p},
B\in \mathcal{F}^{m\times q},C\in \mathcal{F}^{s\times p}$ and $D\in \mathcal{F}^{t\times p}$ are  arbitrary  general matrices over an arbitrary  division ring $\mathcal{F}$. Then the explicit expressions of $r_{8},\ldots,r_{14}$ that appear in Theorem \ref{lemma01} are given by:
\begin{align*}
r_{8}=-\mathbf{r}\begin{pmatrix}A\\D\end{pmatrix}-\mathbf{r}\begin{pmatrix}A&B\\C&0\end{pmatrix}
+\mathbf{r}\begin{pmatrix}0&A&B\\A&A&0\\C&0&0\\0&D&0\end{pmatrix},
\end{align*}
\begin{align*}
r_{9}=\mathbf{r}\begin{pmatrix}A&B\\C&0\end{pmatrix}+\mathbf{r}\begin{pmatrix}A&B\\D&0\end{pmatrix}-\mathbf{r}\begin{pmatrix}0&A&B\\A&A&0\\C&0&0\\0&D&0\end{pmatrix}
+\mathbf{r}(A)-\mathbf{r}(A,~B),
\end{align*}
\begin{align*}
r_{10}=-\mathbf{r}(C)-\mathbf{r}\begin{pmatrix}A&B\\D&0\end{pmatrix}+\mathbf{r}\begin{pmatrix}A&A&B\\0&D&0\end{pmatrix},
\end{align*}
\begin{align*}
r_{11}=\mathbf{r}(C)+\mathbf{r}(D)+\mathbf{r}(A,~B)-\mathbf{r}\begin{pmatrix}A&A&B\\C&0&0\\0&D&0\end{pmatrix},
\end{align*}
\begin{align*}
r_{12}=\mathbf{r}\begin{pmatrix}A\\C\end{pmatrix}+\mathbf{r}\begin{pmatrix}A\\D\end{pmatrix}-\mathbf{r}\begin{pmatrix}A\\C\\D\end{pmatrix}
+\mathbf{r}\begin{pmatrix}A&B\\C&0\\D&0\end{pmatrix}-\mathbf{r}\begin{pmatrix}0&A&B\\A&A&0\\0&D&0\end{pmatrix},
\end{align*}
\begin{align*}
r_{13}=\mathbf{r}\begin{pmatrix}A&B\\C&0\end{pmatrix}+\mathbf{r}\begin{pmatrix}A&B\\D&0\end{pmatrix}
+\mathbf{r}\begin{pmatrix}A&A\\C&0\\0&D\end{pmatrix}-\mathbf{r}\begin{pmatrix}A&A&B\\C&0&0\\0&D&0\end{pmatrix}
-\mathbf{r}\begin{pmatrix}0&A&B\\A&A&0\\C&0&0\\0&D&0\end{pmatrix},
\end{align*}
\begin{align*}
r_{14}=\mathbf{r}\begin{pmatrix}C\\D\end{pmatrix}-\mathbf{r}\begin{pmatrix}A&B\\C&0\\D&0\end{pmatrix}
-\mathbf{r}\begin{pmatrix}A&A\\C&0\\0&D\end{pmatrix}+\mathbf{r}\begin{pmatrix}0&A&B\\A&A&0\\C&0&0\\0&D&0\end{pmatrix}.
\end{align*}
\end{theorem}

\begin{proof}Note that we can not derive the explicit expressions of  $r_{8},\ldots,r_{14}$ from the equalities  in (\ref{equhh028})-(\ref{equhh029}). In order to solve this problem, we need to construct some more complicated block matrices and investigate some relations.
We construct the following three block matrices:
\begin{align}
\begin{pmatrix}A&A\\C&0\\0&D\end{pmatrix},\quad \begin{pmatrix}A&A&B\\C&0&0\\0&D&0\end{pmatrix},
\quad \begin{pmatrix}0&A&B\\A&A&0\\C&0&0\\0&D&0\end{pmatrix}.
\end{align}
We now turn our attention to the ranks of the above three block matrices. We want to give the ranks of these matrices in terms of $r_{1},\ldots,r_{14},r_{\pi},r_{\theta}$. Upon computations, we obtain
\begin{align}\label{xin022}
\mathbf{r}\begin{pmatrix}A&A\\C&0\\0&D\end{pmatrix}=r_{3}+2r_{4}+2r_{5}+r_{6}+r_{7}+2r_{8}+2r_{10}+r_{12}+2r_{13}+r_{14}+r_{\pi}+r_{\theta},
\end{align}
\begin{align}
&\mathbf{r}\begin{pmatrix}A&A&B\\C&0&0\\0&D&0\end{pmatrix}=\nonumber\\
&\mathbf{r}(B)-r_{2}+r_{3}+2r_{4}+2r_{5}+r_{6}+r_{7}+2r_{8}+r_{9}+2r_{10}+r_{12}+r_{13}+r_{14}+r_{\pi}+r_{\theta},
\end{align}
\begin{align}\label{xin024}
&\mathbf{r}\begin{pmatrix}0&A&B\\A&A&0\\C&0&0\\0&D&0\end{pmatrix}=\nonumber\\
&\mathbf{r}(B)-r_{2}+r_{3}+3r_{4}+2r_{5}+r_{6}+r_{7}+3r_{8}+2r_{10}+2r_{12}+2r_{13}+2r_{14}+2r_{\pi}+2r_{\theta}.
\end{align}
We can derive the explicit expressions of $r_{8},\ldots,r_{14}$ via equations (\ref{xin110})-(\ref{equhh029}) and (\ref{xin022})-(\ref{xin024}).

\end{proof}

Similarly, we can obtain the corresponding results about the matrix array:
\begin{align}
\bordermatrix{
~& m_{1} & s_{1}&t_{1} \cr
p_{1}&E&F&G\cr
q_{1}&H& & },
\end{align}
where $E\in \mathcal{F}^{p_{1}\times m_{1}},
F\in \mathcal{F}^{p_{1}\times s_{1}},G\in \mathcal{F}^{p_{1}\times t_{1}}$, and $H\in \mathcal{F}^{q_{1}\times m_{1}}$.

\begin{theorem}\label{theorem02}
  Given $E\in \mathcal{F}^{p_{1}\times m_{1}},
F\in \mathcal{F}^{p_{1}\times s_{1}},G\in \mathcal{F}^{p_{1}\times t_{1}}$, and $H\in \mathcal{F}^{q_{1}\times m_{1}}$.  Then, there exist
 $M_{1}\in GL_{m_{1}}(\mathcal{F}),~P_{1}\in GL_{p_{1}}(\mathcal{F}),~
Q_{1}\in GL_{q_{1}}(\mathcal{F}),~S_{1}\in GL_{s_{1}}(\mathcal{F})$, and $T_{1}\in GL_{t_{1}}(\mathcal{F})$ such that
\begin{align}\label{equ221}
P_{1}EM_{1}=S_{e},\quad P_{1}FS_{1}=S_{f},\quad P_{1}GT_{1}=S_{g},\quad Q_{1}HM_{1}=S_{h},
\end{align}
where
\begin{align}\label{equ222}
S_{e}=
\bordermatrix{
~& v_{1} &v_{3} &v_{2} \cr
v_{2}&0&0&I&0\cr
v_{1}&I&0&0&0\cr
&0&0&0&0 },~
S_{f}=
\bordermatrix{
~& v_{4}&v_{6}&v_{8}\cr
v_{4}&I&0&0&0\cr
v_{5}&0&0&0&0\cr
v_{6}&0&I&0&0\cr
v_{7}&0&0&0&0\cr
v_{8}&0&0&I&0\cr
&0&0&0&0 },~
S_{h}=\bordermatrix{
~& \mathbf{r}(H)\cr
\mathbf{r}(H)&I&0\cr
&0&0},
\end{align}
\begin{align*}
S_{g}=
\begin{matrix}v_{9}\\v_{4}-v_{9}\\v_{10}\\v_{11}\\v_{12}\\v_{5}-v_{10}-v_{11}-v_{12}\\v_{11}\\v_{13}\\v_{6}-v_{11}-v_{13}\\
v_{14}\\v_{15}\\v_{7}-v_{14}-v_{15}\\v_{15}\\v_{12}\\v_{16}\\v_{8}-v_{15}-v_{12}-v_{16}\\v_{17}\\~\end{matrix}\begin{pmatrix}
I&0&0&0&0&0&0&0&0&0\\
0&0&0&0&0&0&0&0&0&0\\
0&I&0&0&0&0&0&0&0&0\\
0&0&I&0&0&0&0&0&0&0\\
0&0&0&0&0&I&0&0&0&0\\
0&0&0&0&0&0&0&0&0&0\\
0&0&I&0&0&0&0&0&0&0\\
0&0&0&I&0&0&0&0&0&0\\
0&0&0&0&0&0&0&0&0&0\\
0&0&0&0&0&0&0&I&0&0\\
0&0&0&0&I&0&0&0&0&0\\
0&0&0&0&0&0&0&0&0&0\\
0&0&0&0&I&0&0&0&0&0\\
0&0&0&0&0&I&0&0&0&0\\
0&0&0&0&0&0&I&0&0&0\\
0&0&0&0&0&0&0&0&0&0\\
0&0&0&0&0&0&0&0&I&0\\
0&0&0&0&0&0&0&0&0&0
\end{pmatrix},
\end{align*}
\begin{align*}
v_{1}=\mathbf{r}(E)+\mathbf{r}(H)-\mathbf{r}\begin{pmatrix}E\\H\end{pmatrix},
v_{2}=\mathbf{r}\begin{pmatrix}E\\H\end{pmatrix}-\mathbf{r}(H),
v_{3}=-\mathbf{r}(E)+\mathbf{r}\begin{pmatrix}E\\H\end{pmatrix},
\end{align*}
\begin{align*}
v_{4}=\mathbf{r}(F)+\mathbf{r}\begin{pmatrix}E\\H\end{pmatrix}-\mathbf{r}\begin{pmatrix}E&F\\H&0\end{pmatrix},
v_{5}=\mathbf{r}\begin{pmatrix}E&F\\H&0\end{pmatrix}-\mathbf{r}(F)-\mathbf{r}(H),
\end{align*}
\begin{align*}
v_{6}=\mathbf{r}\begin{pmatrix}E&F\\H&0\end{pmatrix}+\mathbf{r}(E)-\mathbf{r}\begin{pmatrix}E\\H\end{pmatrix}-\mathbf{r}(E,~F),
v_{7}=\mathbf{r}(E,~F)+\mathbf{r}(H)-\mathbf{r}\begin{pmatrix}E&F\\H&0\end{pmatrix},
\end{align*}
\begin{align*}
v_{8}=\mathbf{r}(E,~F)-\mathbf{r}(E),~v_{9}=\mathbf{r}(F)+\mathbf{r}(G)+\mathbf{r}\begin{pmatrix}E\\H\end{pmatrix}-\mathbf{r}\begin{pmatrix}E&F&0\\E&0&G\\H&0&0\end{pmatrix},
\end{align*}
\begin{align*}
v_{10}=-\mathbf{r}(F)-\mathbf{r}\begin{pmatrix}E&G\\H&0\end{pmatrix}+\mathbf{r}\begin{pmatrix}E&F&0\\E&0&G\\H&0&0\end{pmatrix},
\end{align*}
\begin{align*}
v_{11}= \mathbf{r}\begin{pmatrix}E&F\\H&0\end{pmatrix}+\mathbf{r}\begin{pmatrix}E&G\\H&0\end{pmatrix}+\mathbf{r}\begin{pmatrix}E&F&0\\E&0&G\end{pmatrix}
-\mathbf{r}\begin{pmatrix}E&F&0\\E&0&G\\H&0&0\end{pmatrix}-\mathbf{r}\begin{pmatrix}0&E&F&0\\E&E&0&G\\H&0&0&0\end{pmatrix},
\end{align*}
\begin{align*}
v_{12}=\mathbf{r}(F,~G)-\mathbf{r}\begin{pmatrix}E&F&G\\H&0&0\end{pmatrix}-\mathbf{r}\begin{pmatrix}E&F&0\\E&0&G\end{pmatrix}
+\mathbf{r}\begin{pmatrix}0&E&F&0\\E&E&0&G\\H&0&0&0\end{pmatrix},
\end{align*}
\begin{align*}
v_{13}=\mathbf{r}(E)-\mathbf{r}\begin{pmatrix}E\\H\end{pmatrix}-\mathbf{r}\begin{pmatrix}E&F&0\\E&0&G\end{pmatrix}
+\mathbf{r}\begin{pmatrix}E&F&0\\E&0&G\\H&0&0\end{pmatrix},
\end{align*}
\begin{align*}
v_{14}=-\mathbf{r}(E,~G)-\mathbf{r}\begin{pmatrix}E&F\\H&0\end{pmatrix}+\mathbf{r}\begin{pmatrix}0&E&F&0\\E&E&0&G\\H&0&0&0\end{pmatrix},
\end{align*}
\begin{align*}
v_{15}=\mathbf{r}(E,~F)+\mathbf{r}(E,~G)-\mathbf{r}(E,~F,~G)+\mathbf{r}\begin{pmatrix}E&F&G\\H&0&0\end{pmatrix}
-\mathbf{r}\begin{pmatrix}0&E&F&0\\E&E&0&G\\H&0&0&0\end{pmatrix},
\end{align*}
\begin{align*}
v_{16}=-\mathbf{r}(E)-\mathbf{r}(F,~G)+\mathbf{r}\begin{pmatrix}E&F&0\\E&0&G\end{pmatrix},
~
v_{17}=-\mathbf{r}(E,~F)+\mathbf{r}(E,~F,~G).
\end{align*}

\end{theorem}

\section{\textbf{Solvability conditions for the systems of matrix equations }}

One of important applications of the simultaneous decompositions discussed in Section 2 is to investigate some systems of matrix equations. In this section, we consider several systems of generalized Sylvester matrix equations over $\mathcal{F}$. Furthermore, assume that $\mathcal{F}$ has an involutive anti-automorphism ``$*$'', Hermitian solutions for the systems (\ref{system03}), (\ref{system04}), (\ref{system06}), and (\ref{system07}) are discussed.

We first consider the following system of generalized Sylvester matrix equations (\ref{system02}) over $\mathcal{F}$.

\begin{theorem}\label{theorem04}
The system of matrix equations (\ref{system02}) is consistent if and only if the following rank equalities hold:
\begin{align}\label{equ031}
\mathbf{r}\begin{pmatrix}\Phi &A\\H&0\end{pmatrix}=\mathbf{r}(A)+\mathbf{r}(H),
\mathbf{r}\begin{pmatrix}\Phi &B\\E&0\end{pmatrix}=\mathbf{r}(B)+\mathbf{r}(E),
\end{align}
\begin{align}\label{hhequ01}
\mathbf{r}(A,~B,~\Phi)=\mathbf{r}(A,~B),
\mathbf{r}\begin{pmatrix}E\\H\\ \Phi \end{pmatrix}=\mathbf{r}\begin{pmatrix}E\\H\end{pmatrix},
\end{align}
\begin{align}\label{equ032}
\mathbf{r}(C,~\Psi)=\mathbf{r}(C),\mathbf{r}(D,~\Omega)=\mathbf{r}(D),\mathbf{r}\begin{pmatrix}F\\ \Psi \end{pmatrix}=\mathbf{r}(F),\mathbf{r}\begin{pmatrix}G\\ \Omega\end{pmatrix}=\mathbf{r}(G),
\end{align}
\begin{align}\label{equ033}
\mathbf{r}\begin{pmatrix}\Psi&0&C\\0&-\Omega&D\\F&G&0\end{pmatrix}=\mathbf{r}\begin{pmatrix}C\\D\end{pmatrix}+\mathbf{r}(F,~G),
\end{align}
\begin{align}
\mathbf{r}\begin{pmatrix}0&0&E&F\\A&B&-\Phi&0\\C&0&0&\Psi\end{pmatrix}=\mathbf{r}\begin{pmatrix}A&B\\C&0\end{pmatrix}+\mathbf{r}(E,~F),
\mathbf{r}\begin{pmatrix}0&E&F\\0&H&0\\A&-\Phi&0\\C&0&\Psi\end{pmatrix}=\mathbf{r}\begin{pmatrix}E&F\\H&0\end{pmatrix}+\mathbf{r}\begin{pmatrix}A\\C\end{pmatrix},
\end{align}
\begin{align}
\mathbf{r}\begin{pmatrix}0&0&E&G\\A&B&-\Phi &0\\ D&0&0&\Omega\end{pmatrix}=\mathbf{r}\begin{pmatrix}A&B\\D&0\end{pmatrix}+\mathbf{r}(E,~G),
\mathbf{r}\begin{pmatrix}0&E&G\\0&H&0\\A&-\Phi&0\\D&0&\Omega\end{pmatrix}=\mathbf{r}\begin{pmatrix}E&G\\H&0\end{pmatrix}+\mathbf{r}\begin{pmatrix}A\\D\end{pmatrix},
\end{align}
\begin{align}
\mathbf{r}\begin{pmatrix}0&0&E&F&G\\0&0&H&0&0\\A&A&-\Phi&0&0\\C&0&0&\Psi&0\\0&D&0&0&\Omega\end{pmatrix}=\mathbf{r}\begin{pmatrix}E&F&G\\H&0&0\end{pmatrix}+
\mathbf{r}\begin{pmatrix}A&A\\C&0\\0&D\end{pmatrix},
\end{align}
\begin{align}
\mathbf{r}\begin{pmatrix}0&0&E&F&0\\0&0&E&0&G\\A&B&-\Phi&0&0\\C&0&0&\Psi&0\\D&0&0&0&\Omega\end{pmatrix}=\mathbf{r}\begin{pmatrix}A&B\\C&0\\D&0\end{pmatrix}
+\mathbf{r}\begin{pmatrix}E&F&0\\E&0&G\end{pmatrix},
\end{align}
\begin{align}
\mathbf{r}\begin{pmatrix}0&0&0&E&F&G\\A&A&B&-\Phi&0&0\\C&0&0&0&\Psi&0\\0&D&0&0&0&\Omega\end{pmatrix}=\mathbf{r}(E,~F,~G)+\mathbf{r}\begin{pmatrix}A&A&B\\C&0&0\\0&D&0\end{pmatrix},
\end{align}
\begin{align}\label{equ039}
\mathbf{r}\begin{pmatrix}0&E&F&0\\0&E&0&G\\0&H&0&0\\A&-\Phi&0&0\\C&0&\Psi&0\\D&0&0&\Omega\end{pmatrix}=\mathbf{r}\begin{pmatrix}A\\C\\D\end{pmatrix}+
\mathbf{r}\begin{pmatrix}E&F&0\\E&0&G\\H&0&0\end{pmatrix},
\end{align}
\begin{align}\label{equ310}
\mathbf{r}\begin{pmatrix}0&0&0&0&0&0&H\\
0&0&0&G&0&E&0\\
0&0&0&0&F&E&E\\
0&C&0&0&-\Psi&0&0\\
0&0&D&\Omega&0&0&0\\
0&A&A&0&0&0&\Phi\\
B&0&A&0&0&-\Phi&0\end{pmatrix}=\mathbf{r}\begin{pmatrix}0&A&B\\A&A&0\\C&0&0\\0&D&0\end{pmatrix}+\mathbf{r}\begin{pmatrix}0&E&F&0\\E&E&0&G\\H&0&0&0\end{pmatrix}.
\end{align}

\end{theorem}

\begin{proof}

$\Longrightarrow:$ Let $(X_{0},Y_{0})$ be a common solution to the system of matrix equations (\ref{system02}), that is
\begin{align}\label{equ311}
\left\{\begin{array}{rll}
AX_{0}E+BY_{0}H & = & \Phi,\\
CX_{0}F & = & \Psi,\\
DX_{0}G & = & \Omega.
\end{array}
  \right.
\end{align}
Now we want to show that the rank equalities in (\ref{equ031})-(\ref{equ310}) hold using the matrix equations (\ref{equ311})
and elementary matrix operations.

\begin{itemize}
  \item For the rank equalities in (\ref{equ031})-(\ref{equ032}). Observe that
  \begin{align*}
  \mathbf{r}\begin{pmatrix}\Phi&A\\H&0\end{pmatrix}=\mathbf{r}\left[\begin{pmatrix}I&BY_{0}\\0&I\end{pmatrix}\begin{pmatrix}AX_{0}E+BY_{0}H&A\\H&0\end{pmatrix}\begin{pmatrix}I&0\\X_{0}E&I\end{pmatrix}\right]
  =\mathbf{r}\begin{pmatrix}0&A\\H&0\end{pmatrix}=\mathbf{r}(A)+\mathbf{r}(H).
  \end{align*}
  Similarly, we can prove the other seven rank equalities in (\ref{equ031})-(\ref{equ032}).
  \item For the rank equalities in (\ref{equ033})-(\ref{equ039}). Observe that
  \begin{align*}
  &\mathbf{r}\begin{pmatrix}\Psi&0&C\\0&-\Omega&D\\F&G&0\end{pmatrix}=\mathbf{r}\begin{pmatrix}CX_{0}F&0&C\\0&-DX_{0}G&D\\F&G&0\end{pmatrix}\\
  &=\mathbf{r}\left[ \begin{pmatrix}I&0&0\\0&I&-DX_{0}\\0&0&I\end{pmatrix}\begin{pmatrix}0&0&C\\0&0&D\\F&G&0\end{pmatrix}\begin{pmatrix}I&0&0\\0&I&0\\X_{0}F&0&I\end{pmatrix}\right]=
  \mathbf{r}\begin{pmatrix}0&0&C\\0&0&D\\F&G&0\end{pmatrix}\\&
  =\mathbf{r}\begin{pmatrix}C\\D\end{pmatrix}+\mathbf{r}(F,~G).
  \end{align*}Similarly, we can prove the other eight rank equalities in (\ref{equ033})-(\ref{equ039}).
  \item For the rank equality (\ref{equ310}). Observe that
  \begin{align*}
  &\mathbf{r}\begin{pmatrix}\begin{smallmatrix}
  0&0&0&0&0&0&H\\
  0&0&0&G&0&E&0\\
  0&0&0&0&F&E&E\\
  0&C&0&0&-\Psi&0&0\\
  0&0&D&\Omega&0&0&0\\
  0&A&A&0&0&0&\Phi\\
  B&0&A&0&0&-\Phi&0
  \end{smallmatrix}\end{pmatrix}=
  \mathbf{r}\begin{pmatrix}\begin{smallmatrix}
  0&0&0&0&0&0&H\\
  0&0&0&G&0&E&0\\
  0&0&0&0&F&E&E\\
  0&C&0&0&-CX_{0}F&0&0\\
  0&0&D&DX_{0}G&0&0&0\\
  0&A&A&0&0&0&AX_{0}E+BY_{0}H\\
  B&0&A&0&0&-AX_{0}E-BY_{0}H&0
  \end{smallmatrix}\end{pmatrix}\\[3mm]
  &
  =\mathbf{r}\left[
  \begin{pmatrix}\begin{smallmatrix}
  I&0&0&0&0&0&0\\
  0&I&0&0&0&0&0\\
  0&0&I&0&0&0&0\\
  0&0&-CX_{0}&I&0&0&0\\
  0&DX_{0}&0&0&I&0&0\\
  BY_{0}&0&0&0&0&I&0\\
  0&0&0&0&0&0&I
  \end{smallmatrix}\end{pmatrix}
  \begin{pmatrix}\begin{smallmatrix}
  0&0&0&0&0&0&H\\
  0&0&0&G&0&E&0\\
  0&0&0&0&F&E&E\\
  0&C&0&0&0&0&0\\
  0&0&D&0&0&0&0\\
  0&A&A&0&0&0&0\\
  B&0&A&0&0&0&0
  \end{smallmatrix}\end{pmatrix}
  \begin{pmatrix}\begin{smallmatrix}
  I&0&0&0&0&-Y_{0}H&0\\
  0&I&0&0&0&X_{0}E&X_{0}E\\
  0&0&I&0&0&-X_{0}E&0\\
  0&0&0&I&0&0&0\\
  0&0&0&0&I&0&0\\
  0&0&0&0&0&I&0\\
  0&0&0&0&0&0&I
  \end{smallmatrix}\end{pmatrix}\right]
  \\[3mm]
  &=
  \mathbf{r}\begin{pmatrix}0&A&B\\A&A&0\\C&0&0\\0&D&0\end{pmatrix}+\mathbf{r}\begin{pmatrix}0&E&F&0\\E&E&0&G\\H&0&0&0\end{pmatrix}.
  \end{align*}
\end{itemize}

$\Longleftarrow:$ It follows from Theorem \ref{theorem01} and \ref{theorem02} that the system of matrix equations (\ref{system02}) is equivalent to the system of matrix equations
\begin{align}
\left\{\begin{array}{rll}
S_{a}(P^{-1}XP_{1}^{-1})S_{e}+S_{b}(Q^{-1}YQ_{1}^{-1})S_{h} & = & M \Phi M_{1},\\
S_{c}(P^{-1}XP_{1}^{-1})S_{f} & = & S \Psi S_{1},\\
S_{d}(P^{-1}XP_{1}^{-1})S_{g} & = &T \Omega T_{1}.
\end{array}
  \right.
\end{align}
Put
\begin{align}\label{hhhequ312}
\widehat{X}=P^{-1}XP_{1}^{-1}:=\left(X_{ij}\right)_{18\times 18},\ \ ~~\widehat{Y}:=Q^{-1}YQ_{1}^{-1}=\left(Y_{ij}\right)_{8\times 8},
\end{align}
\begin{align}\label{hhhequ313}
M \Phi M_{1}:=\left(\Phi_{ij}\right)_{14\times 14},\  \ S \Psi S_{1}:=\left(\Psi_{ij}\right)_{10\times 10},\ \ T \Omega T_{1}:=\left(\Omega_{ij}\right)_{10\times 10},
\end{align}where the block rows of $\left(\Phi_{ij}\right)_{14\times 14},\left(\Psi_{ij}\right)_{10\times 10},$ and $\left(\Omega_{ij}\right)_{10\times 10}$ are the same as the block rows of $S_{a},S_{c},$ and $S_{d}$, the block columns of $\left(\Phi_{ij}\right)_{14\times 14},\left(\Psi_{ij}\right)_{10\times 10},$ and $\left(\Omega_{ij}\right)_{10\times 10}$ are the same as the block columns of $S_{e},S_{f},$ and $S_{g}$.
Then it follows from (\ref{equ221}), (\ref{equ222}), (\ref{hhhequ312}), and (\ref{hhhequ313}) that
\begin{align}\label{equ315}
S_{a}\widehat{X}S_{e}+S_{b}\widehat{Y}S_{h}=
\begin{pmatrix}
\widehat{\Phi}_{1}&\widehat{\Phi}_{2}\\ \widehat{\Phi}_{3}&\widehat{\Phi}_{4}
\end{pmatrix},\ \ S_{d}\widehat{X}S_{g}=
\begin{pmatrix}
\widehat{\Omega}_{1}&\widehat{\Omega}_{2}\\ \widehat{\Omega}_{3}&\widehat{\Omega}_{4}
\end{pmatrix},
\end{align}
\begin{align}\label{equ316}
S_{c}\widehat{X}S_{f}=
\begin{pmatrix}
X_{11}&X_{12}&X_{17}&X_{18}&X_{19}&X_{1,13}&X_{1,14}&X_{1,15}&X_{1,16}&0\\
X_{21}&X_{22}&X_{27}&X_{28}&X_{29}&X_{2,13}&X_{2,14}&X_{2,15}&X_{2,16}&0\\
X_{71}&X_{72}&X_{77}&X_{78}&X_{79}&X_{7,13}&X_{7,14}&X_{7,15}&X_{7,16}&0\\
X_{81}&X_{82}&X_{87}&X_{88}&X_{89}&X_{8,13}&X_{8,14}&X_{8,15}&X_{8,16}&0\\
X_{91}&X_{92}&X_{97}&X_{98}&X_{99}&X_{9,13}&X_{9,14}&X_{9,15}&X_{9,16}&0\\
X_{13,1}&X_{13,2}&X_{13,7}&X_{13,8}&X_{13,9}&X_{13,13}&X_{13,14}&X_{13,15}&X_{13,16}&0\\
X_{14,1}&X_{14,2}&X_{14,7}&X_{14,8}&X_{14,9}&X_{14,13}&X_{14,14}&X_{14,15}&X_{14,16}&0\\
X_{15,1}&X_{15,2}&X_{15,7}&X_{15,8}&X_{15,9}&X_{15,13}&X_{15,14}&X_{15,15}&X_{15,16}&0\\
X_{16,1}&X_{16,2}&X_{16,7}&X_{16,8}&X_{16,9}&X_{16,13}&X_{16,14}&X_{16,15}&X_{16,16}&0\\
0&0&0&0&0&0&0&0&0&0
\end{pmatrix},
\end{align}
where
\begin{align}
\widehat{\Phi}_{1}=\begin{pmatrix}
X_{77}+Y_{11}&X_{78}+Y_{12}&X_{79}+Y_{13}&X_{7,10}+Y_{14}&X_{7,11}+Y_{15}&X_{7,12}+Y_{16}&Y_{17}\\
X_{87}+Y_{21}&X_{88}+Y_{22}&X_{89}+Y_{23}&X_{8,10}+Y_{24}&X_{8,11}+Y_{25}&X_{8,12}+Y_{26}&Y_{27}\\
X_{97}+Y_{31}&X_{98}+Y_{32}&X_{99}+Y_{33}&X_{9,10}+Y_{34}&X_{9,11}+Y_{35}&X_{9,12}+Y_{36}&Y_{37}\\
X_{10,7}+Y_{41}&X_{10,8}+Y_{42}&X_{10,9}+Y_{43}&X_{10,10}+Y_{44}&X_{10,11}+Y_{45}&X_{10,12}+Y_{46}&Y_{47}\\
X_{11,7}+Y_{51}&X_{11,8}+Y_{52}&X_{11,9}+Y_{53}&X_{11,10}+Y_{54}&X_{11,11}+Y_{55}&X_{11,12}+Y_{56}&Y_{57}\\
X_{12,7}+Y_{61}&X_{12,8}+Y_{62}&X_{12,9}+Y_{63}&X_{12,10}+Y_{64}&X_{12,11}+Y_{65}&X_{12,12}+Y_{66}&Y_{45}\\
Y_{71}&Y_{72}&Y_{73}&Y_{74}&Y_{75}&Y_{76}&Y_{77}
\end{pmatrix},
\end{align}
\begin{align}
\widehat{\Phi}_{2}=
\begin{pmatrix}
X_{71}&\cdots&X_{76}&0\\
\vdots&\vdots&\ddots&\vdots\\
X_{12,1}&\cdots&X_{12,6}&0\\
0&\cdots&0&0
\end{pmatrix},
\widehat{\Phi}_{3}=
\begin{pmatrix}
X_{17}&\cdots&X_{1,12}&0\\
\vdots&\ddots&\vdots\\
X_{67}&\cdots&X_{6,12}&0\\
0&\cdots&0&0
\end{pmatrix},
\widehat{\Phi}_{4}=
\begin{pmatrix}
X_{11}&\cdots&X_{16}&0\\
\vdots&\ddots&\vdots&0\\
X_{61}&\cdots&X_{66}&0\\
0&\cdots&0&0
\end{pmatrix},
\end{align}
\begin{align}
\widehat{\Omega}_{1}=
\begin{pmatrix}
X_{11}&X_{13}&X_{14}+X_{17}&X_{18}&X_{1,11}+X_{1,13}\\
X_{31}&X_{33}&X_{34}+X_{37}&X_{38}&X_{3,11}+X_{3,13}\\
X_{41}+X_{71}&X_{43}+X_{73}&\begin{matrix}X_{44}+X_{74}\\+X_{47}+X_{77}\end{matrix}&X_{48}+X_{78}&\begin{matrix}X_{4,11}+X_{7,11}\\+X_{4,13}+X_{7,13}\end{matrix}\\
X_{81}&X_{83}&X_{84}+X_{87}&X_{88}&X_{8,11}+X_{8,13}\\
X_{11,1}+X_{13,1}&X_{11,3}+X_{13,3}&\begin{matrix}X_{11,4}+X_{13,4}\\+X_{11,7}+X_{13,7}\end{matrix}&X_{11,8}+X_{13,8}&\begin{matrix}X_{11,,11}+X_{13,11}\\+X_{11,,13}+X_{13,13}\end{matrix}
\end{pmatrix},
\end{align}
\begin{align}
\widehat{\Omega}_{2}=
\begin{pmatrix}
X_{15}+X_{1,14}&X_{1,15}&X_{1,10}&X_{1,17}&0\\
X_{35}+X_{3,14}&X_{3,15}&X_{3,10}&X_{3,17}&0\\
\begin{matrix}X_{45}+X_{75}\\+X_{4,14}+X_{7,14}\end{matrix}&X_{4,15}+X_{7,15}&X_{4,10}+X_{7,10}&X_{4,17}+X_{7,17}&0\\
X_{85}+X_{8,14}&X_{8,15}&X_{8,10}&X_{8,17}&0\\
\begin{matrix}X_{11,5}+X_{13,5}\\+X_{11,14}+X_{13,14}\end{matrix}&X_{11,15}+X_{13,15}&X_{11,10}+X_{13,10}&X_{11,17}+X_{13,17}&0
\end{pmatrix},
\end{align}
\begin{align}
\widehat{\Omega}_{3}=
\begin{pmatrix}
X_{51}+X_{14,1}&X_{53}+X_{14,3}&\begin{matrix}X_{54}+X_{14,4}\\+X_{57}+X_{14,7}\end{matrix}&X_{58}+X_{14,8}&\begin{matrix}X_{5,11}+X_{14,11}\\+X_{5,13}+X_{14,13}\end{matrix}\\
X_{15,1}&X_{15,3}&X_{15,4}+X_{15,7}&X_{15,8}&X_{15,11}+X_{15,13}\\
X_{10,1}&X_{10,3}&X_{10,4}+X_{10,7}&X_{10,8}&X_{10,11}+X_{10,13}\\
X_{17,1}&X_{17,3}&X_{17,4}+X_{17,7}&X_{17,8}&X_{17,11}+X_{17,13}\\
0&0&0&0&0
\end{pmatrix},
\end{align}
\begin{align}
\widehat{\Omega}_{4}=
\begin{pmatrix}
X_{55}+X_{14,5}+X_{5,14}+X_{14,14}&X_{5,15}+X_{14,15}&X_{5,10}+X_{14,10}&X_{5,17}+X_{14,17}&0\\
X_{15,5}+X_{15,14}&X_{15,15}&X_{15,10}&X_{15,17}&0\\
X_{10,5}+X_{10,14}&X_{10,15}&X_{10,10}&X_{10,17}&0\\
X_{17,5}+X_{17,14}&X_{17,15}&X_{17,10}&X_{17,17}&0\\
0&0&0&0&0
\end{pmatrix}.
\end{align}
Upon computations, we obtain that the matrix equations in (\ref{equ315}) and (\ref{equ316}) have a common solution  if and only
if
\begin{align}\label{equ324}
\begin{pmatrix}\Phi_{1,14}\\ \Phi_{2,14}\\ \vdots \\ \Phi_{14,14}\end{pmatrix}=0,~(\Phi_{14,1},\Phi_{14,2},\ldots,\Phi_{14,13})=0,
\end{align}
\begin{align}\label{equ325}
(\Phi_{78},\Phi_{79},\ldots,\Phi_{7,13})=0,~\begin{pmatrix}\Phi_{87}\\ \Phi_{97}\\ \vdots \\ \Phi_{13,7}\end{pmatrix}=0,
\end{align}
\begin{align}\label{equ326}
\begin{pmatrix}\Psi_{1,10}\\ \Psi_{2,10}\\ \vdots \\ \Psi_{10,10}\end{pmatrix}=0,~(\Psi_{10,1},\Psi_{10,2},\ldots,\Psi_{10,9})=0,
\end{align}
\begin{align}\label{equ327}
\begin{pmatrix}\Omega_{1,10} \\ \Omega_{2,10} \\ \vdots \\ \Omega_{10,10} \end{pmatrix}=0,~(\Omega_{10,1},\Omega_{10,2},\ldots,\Omega_{10,9})=0,
\end{align}
\begin{align}\label{equ332}
\Psi_{18}=\Omega_{17},\Psi_{81}=\Omega_{71},\Psi_{48}=\Omega_{47},\Psi_{84}=\Omega_{74},\Psi_{44}=\Omega_{44},\Psi_{88}=\Omega_{77},
\end{align}
\begin{align}\label{equ333}
\left\{\begin{array}{c}
\Phi_{89}=\Psi_{12},\Phi_{98}=\Psi_{21},\Phi_{99}=\Psi_{22},\Phi_{81}=\Psi_{13},\\
\Phi_{83}=\Psi_{15},\Phi_{91}=\Psi_{23},\Phi_{92}=\Psi_{24},\Phi_{93}=\Psi_{25},
\end{array}
  \right.
\end{align}
\begin{align}\label{equ334}
\Phi_{18}=\Psi_{31},\Phi_{38}=\Psi_{51},\Phi_{19}=\Psi_{32},\Phi_{29}=\Psi_{42},\Phi_{39}=\Psi_{52},
\end{align}
\begin{align}\label{equ335}
\left\{\begin{array}{c}
\Phi_{8,10}=\Omega_{12},\Phi_{10,8}=\Omega_{21},\Phi_{10,2}=\Omega_{24},\Phi_{84}=\Omega_{18},\\
\Phi_{10,4}=\Omega_{28},\Phi_{10,10}=\Omega_{22},\Phi_{10,11}+\Phi_{10,1}=\Omega_{23},
\end{array}
  \right.
\end{align}
\begin{align}\label{equrrrr336}
\Phi_{2,10}=\Omega_{42},\Phi_{48}=\Omega_{81},\Phi_{4,10}=\Omega_{82},\Phi_{11,8}+\Phi_{18}=\Omega_{31},\Phi_{11,10}+\Phi_{1,10}=\Omega_{32},
\end{align}
\begin{align}\label{equrrrr337}
\left\{\begin{array}{c}
\Phi_{88}=\Psi_{11}=\Omega_{11},\Phi_{28}=\Psi_{41}=\Omega_{41},\\
\Phi_{8,11}+\Phi_{81}=\Omega_{13},\Phi_{2,12}+\Psi_{47}=\Omega_{46},\Phi_{2,11}+\Psi_{43}=\Omega_{43},
\end{array}
  \right.
\end{align}
\begin{align}\label{equ338}
\Phi_{82}=\Psi_{14}=\Omega_{14},\Phi_{11,2}+\Psi_{34}=\Omega_{34},\Phi_{12,2}+\Psi_{74}=\Omega_{64},
\end{align}
\begin{align}\label{equ339}
\Phi_{85}+\Psi_{16}=\Omega_{15},
\Phi_{58}+\Psi_{61}=\Omega_{51},
\Phi_{8,12}+\Psi_{17}=\Omega_{16},
\Phi_{12,8}+\Psi_{71}=\Omega_{61},
\end{align}
\begin{align}\label{equ340}
\quad \Phi_{1,11}+\Phi_{11,1}+\Phi_{11,11}+\Psi_{33}=\Omega_{33}.
\end{align}
Now we want to prove that (\ref{equ031})-(\ref{equ310}) $\Longrightarrow$ (\ref{equ324})-(\ref{equ340}). First, we show that (\ref{equ031}) and (\ref{hhequ01}) $\Longrightarrow (\ref{equ324}) ~\mbox{and} ~(\ref{equ325}).$ Upon computations, we obtain
\begin{align*}
&\mathbf{r}\begin{pmatrix}\Phi&A\\H&0\end{pmatrix}=\mathbf{r}(A)+\mathbf{r}(H)\Longleftrightarrow \mathbf{r}\begin{pmatrix}(\Phi_{ij})_{14\times 14}&S_{a}\\S_{h}&0\end{pmatrix}=\mathbf{r}(S_{a})+\mathbf{r}(S_{h})\\
&\Longrightarrow
\begin{pmatrix}\Phi_{78}&\Phi_{79}&\cdots&\Phi_{7,14}\\  \Phi_{14,8}&\Phi_{14,9}&\cdots &\Phi_{14,14}\end{pmatrix}=0,
\end{align*}
\begin{align*}
\mathbf{r}\begin{pmatrix}\Phi &B\\E&0\end{pmatrix}=\mathbf{r}(B)+\mathbf{r}(E)\Longleftrightarrow \mathbf{r}\begin{pmatrix}(\Phi_{ij})_{14\times 14}&S_{b}\\S_{e}&0\end{pmatrix}=\mathbf{r}(B)+\mathbf{r}(E)\Longrightarrow
\begin{pmatrix}\Phi_{87}&\Phi_{8,14}\\ \Phi_{97}&\Phi_{9,14}\\ \vdots & \vdots \\ \Phi_{14,7}& \Phi_{14,14}\end{pmatrix}=0,
\end{align*}
\begin{align*}
\mathbf{r}(A,~B,~\Phi)=\mathbf{r}(A,~B) \Longleftrightarrow \mathbf{r}(S_{a},S_{b},(\Phi_{ij})_{14\times 14})=\mathbf{r}(S_{a},S_{b})\Longrightarrow (\Phi_{14,1},\Phi_{14,2},\cdots,\Phi_{14,14})=0,
\end{align*}
\begin{align*}
\mathbf{r}\begin{pmatrix}E\\H\\ \Phi \end{pmatrix}=\mathbf{r}\begin{pmatrix}E\\H\end{pmatrix} \Longleftrightarrow \mathbf{r}\begin{pmatrix}S_{e}\\S_{h}\\ (\Phi_{ij})_{14\times 14}\end{pmatrix}\Longrightarrow
\begin{pmatrix}\Phi_{1,14}\\ \Phi_{2,14}\\ \vdots \\ \Phi_{14,14} \end{pmatrix}=0.
\end{align*}
Similarly, it can be found that
\begin{align*}
\mathbf{r}(C,~\Psi)=\mathbf{r}(C),\mathbf{r}(D,~\Omega)=\mathbf{r}(D),\mathbf{r}\begin{pmatrix}F\\ \Psi \end{pmatrix}=\mathbf{r}(F),\mathbf{r}\begin{pmatrix}G\\ \Omega\end{pmatrix}=\mathbf{r}(G)  \Longrightarrow (\ref{equ326})~ \mbox{and} ~(\ref{equ327}),
\end{align*}
\begin{align*}
\mathbf{r}\begin{pmatrix}\Psi&0&C\\0&-\Omega&D\\F&G&0\end{pmatrix}=\mathbf{r}\begin{pmatrix}C\\D\end{pmatrix}+\mathbf{r}(F,~G) \Longrightarrow (\ref{equ332}),
\end{align*}
\begin{align*}
\mathbf{r}\begin{pmatrix}0&0&E&F\\A&B&-\Phi&0\\C&0&0&\Psi\end{pmatrix}=\mathbf{r}\begin{pmatrix}A&B\\C&0\end{pmatrix}+\mathbf{r}(E,~F)\Longrightarrow (\ref{equ333}),
\end{align*}
\begin{align*}
\mathbf{r}\begin{pmatrix}0&E&F\\0&H&0\\A&-\Phi&0\\C&0&\Psi\end{pmatrix}=\mathbf{r}\begin{pmatrix}E&F\\H&0\end{pmatrix}+\mathbf{r}\begin{pmatrix}A\\C\end{pmatrix}\Longrightarrow (\ref{equ334}),
\end{align*}
\begin{align*}
\mathbf{r}\begin{pmatrix}0&0&E&G\\A&B&-\Phi &0\\ D&0&0&\Omega\end{pmatrix}=\mathbf{r}\begin{pmatrix}A&B\\D&0\end{pmatrix}+\mathbf{r}(E,~G)\Longrightarrow (\ref{equ335}),
\end{align*}
\begin{align*}
\mathbf{r}\begin{pmatrix}0&E&G\\0&H&0\\A&-\Phi&0\\D&0&\Omega\end{pmatrix}=\mathbf{r}\begin{pmatrix}E&G\\H&0\end{pmatrix}+\mathbf{r}\begin{pmatrix}A\\D\end{pmatrix}
\Longrightarrow (\ref{equrrrr336}),
\end{align*}
\begin{align*}
\mathbf{r}\begin{pmatrix}0&0&E&F&G\\0&0&H&0&0\\A&A&-\Phi&0&0\\C&0&0&\Psi&0\\0&D&0&0&\Omega\end{pmatrix}=\mathbf{r}\begin{pmatrix}E&F&G\\H&0&0\end{pmatrix}+
\mathbf{r}\begin{pmatrix}A&A\\C&0\\0&D\end{pmatrix} \Longrightarrow (\ref{equrrrr337}),
\end{align*}
\begin{align*}
\mathbf{r}\begin{pmatrix}0&0&E&F&0\\0&0&E&0&G\\A&B&-\Phi&0&0\\C&0&0&\Psi&0\\D&0&0&0&\Omega\end{pmatrix}=\mathbf{r}\begin{pmatrix}A&B\\C&0\\D&0\end{pmatrix}
+\mathbf{r}\begin{pmatrix}E&F&0\\E&0&G\end{pmatrix}\Longrightarrow (\ref{equ338}),
\end{align*}
\begin{align*}
\left\{\begin{array}{c}
\mathbf{r}\begin{pmatrix}0&0&0&E&F&G\\A&A&B&-\Phi&0&0\\C&0&0&0&\Psi&0\\0&D&0&0&0&\Omega\end{pmatrix}=\mathbf{r}(E,~F,~G)+\mathbf{r}\begin{pmatrix}A&A&B\\C&0&0\\0&D&0\end{pmatrix}\\
\mathbf{r}\begin{pmatrix}0&E&F&0\\0&E&0&G\\0&H&0&0\\A&-\Phi&0&0\\C&0&\Psi&0\\D&0&0&\Omega\end{pmatrix}=\mathbf{r}\begin{pmatrix}A\\C\\D\end{pmatrix}+
\mathbf{r}\begin{pmatrix}E&F&0\\E&0&G\\H&0&0\end{pmatrix}
\end{array}
  \right.\Longrightarrow (\ref{equ339}),
\end{align*}
\begin{align*}
\mathbf{r}\begin{pmatrix}0&0&0&0&0&0&H\\
0&0&0&G&0&E&0\\
0&0&0&0&F&E&E\\
0&C&0&0&-\Psi&0&0\\
0&0&D&\Omega&0&0&0\\
0&A&A&0&0&0&\Phi\\
B&0&A&0&0&-\Phi&0\end{pmatrix}=\mathbf{r}\begin{pmatrix}0&A&B\\A&A&0\\C&0&0\\0&D&0\end{pmatrix}+\mathbf{r}\begin{pmatrix}0&E&F&0\\E&E&0&G\\H&0&0&0\end{pmatrix}\Longrightarrow (\ref{equ340}).
\end{align*}

\end{proof}

\begin{remark}
Wang, Zhang, and van der Woude \cite{zhangxia} gave some necessary and sufficient conditions for the existence of the general solution to (\ref{system02}) using the submatrices of
$\left(\Phi_{ij}\right)_{14\times 14},\left(\Psi_{ij}\right)_{10\times 10},$ and $\left(\Omega_{ij}\right)_{10\times 10}$. The conditions presented in Theorem \ref{theorem04} are more straightforward than the conditions in \cite{zhangxia}. We only need to check whether  these rank equalities in  (\ref{equ031})-(\ref{equ310}) hold or not.

\end{remark}

Next we assume that $\mathcal{F}$ stands for an arbitrary division ring with an involutive anti-automorphism ``$*$'', and the characteristic of $\mathcal{F}\neq2.$ Now we consider the following two systems:
\begin{align}
\label{system03}
\left\{\begin{array}{rll}
AXA^{*}+BYB^{*} & = & \Phi,\\
CXC^{*} & = & \Psi,\\
DXD^{*} & = & \Omega
\end{array}
  \right.
\end{align}
and
\begin{align}
\label{system04}
\left\{
\begin{array}{rll}
AXA^{*}+BYB^{*} & = & \Phi,\\
CXD & = & \Omega
\end{array}
  \right.
\end{align}
where $X$ and $Y$ are unknowns and the others are matrices over $\mathcal{F}$ with compatible dimensions.

\begin{theorem}\label{theorem05}
Given $A\in \mathcal{F}^{m\times p},
B\in \mathcal{F}^{m\times q},C\in \mathcal{F}^{s\times p},D\in \mathcal{F}^{t\times p},\Phi=\Phi^{*}\in \mathcal{F}^{m\times m},\Psi=\Psi^{*}\in \mathcal{F}^{s\times s},$ and $\Omega=\Omega^{*}\in \mathcal{F}^{t\times t}$. The system of matrix equations (\ref{system03}) has a pair of Hermitian solution $(X,Y)\in \mathcal{F}^{p\times p}\times \mathcal{F}^{q\times q}$ if and only if the ranks satisfy:
\begin{align*}
\mathbf{r}\begin{pmatrix}\Phi &A \\ B^{*}& 0\end{pmatrix}=\mathbf{r}(A)+\mathbf{r}(B),\mathbf{r}(A,~B,~\Phi)=\mathbf{r}(A,~B),
\end{align*}
\begin{align*}
\mathbf{r}(C,~\Psi)=\mathbf{r}(C),\mathbf{r}(D,~\Omega)=\mathbf{r}(D),~
\mathbf{r}\begin{pmatrix}\Psi &0&C\\0 &-\Omega &D\\C^{*}&D^{*}&0\end{pmatrix}=2\mathbf{r}\begin{pmatrix}C\\D\end{pmatrix},
\end{align*}
\begin{align*}
\mathbf{r}\begin{pmatrix}0&0&A^{*}&C^{*}\\A&B&-\Phi&0\\C&0&0&\Psi\end{pmatrix}=\mathbf{r}\begin{pmatrix}A&B\\C&0\end{pmatrix}+\mathbf{r}\begin{pmatrix}A\\C\end{pmatrix},
\end{align*}
\begin{align*}
\mathbf{r}\begin{pmatrix}0&0&A^{*}&D^{*}\\A&B&-\Phi&0\\D&0&0&\Omega\end{pmatrix}=\mathbf{r}\begin{pmatrix}A&B\\D&0\end{pmatrix}+\mathbf{r}\begin{pmatrix}A\\D\end{pmatrix},
\end{align*}
\begin{align*}
\mathbf{r}\begin{pmatrix}0&0&A^{*}&C^{*}&D^{*}\\0&0&B^{*}&0&0\\A&A&-\Phi&0&0\\C&0&0&\Psi&0\\0&D&0&0&\Omega\end{pmatrix}=\mathbf{r}\begin{pmatrix}A&B\\C&0\\D&0\end{pmatrix}
+\mathbf{r}\begin{pmatrix}A&A\\C&0\\0&D\end{pmatrix},
\end{align*}
\begin{align*}
\mathbf{r}\begin{pmatrix}0&0&0&A^{*}&C^{*}&D^{*}\\
A&A&B&-\Phi&0&0\\
C&0&0&0&\Psi&0\\
0&D&0&0&0&\Omega\end{pmatrix}=\mathbf{r}\begin{pmatrix}A\\C\\D\end{pmatrix}+\mathbf{r}\begin{pmatrix}A&A&B\\C&0&0\\0&D&0\end{pmatrix},
\end{align*}
\begin{align*}
\mathbf{r}\begin{pmatrix}0&0&0&0&0&0&B^{*}\\
0&0&0&D^{*}&0&A^{*}&0\\
0&0&0&0&C^{*}&A^{*}&A^{*}\\
0&C&0&0&-\Psi&0&0\\
0&0&D&\Omega&0&0&0\\
0&A&A&0&0&0&\Phi\\
B&0&A&0&0&-\Phi&0
\end{pmatrix}=2\mathbf{r}\begin{pmatrix}0&A&B\\A&A&0\\C&0&0\\0&D&0\end{pmatrix}.
\end{align*}
\end{theorem}

\begin{proof}
We first prove that the system of matrix equations (\ref{system03}) has a pair of Hermitian solution $(X,Y)\in \mathcal{F}^{p\times p}\times \mathcal{F}^{q\times q}$
if and only if the system of matrix equations
\begin{align}
\label{equ336}
\left\{\begin{array}{rll}
A\widetilde{X}A^{*}+B\widetilde{Y}B^{*} & = &\Phi,\\
C\widetilde{X}C^{*} & = & \Psi,\\
D\widetilde{X}D^{*} & = & \Omega
\end{array}
  \right.
\end{align}
has a solution. If the system (\ref{system03}) has a pair of Hermitian solution,
say, $(X_{0},Y_{0})$, then the system (\ref{equ336}) clearly
has a solution $(\widetilde{X},\widetilde{Y})=(X_{0},Y_{0})$. Conversely, if the system (\ref{equ336}) has a solution
$(\widetilde{X},\widetilde{Y})$, then
\begin{align*}
(X,Y)=\Big(\frac{\widetilde{X}+\widetilde{X}^{*}}{2},\frac{\widetilde{Y}+\widetilde{Y}^{*}}{2}\Big)
\end{align*}
is a pair Hermitian solution of (\ref{system03}). We can derive the
solvability conditions to the system of matrix equations (\ref{system03}) by
Theorem \ref{theorem04}.
\end{proof}

\begin{theorem}
Let $A\in \mathcal{F}^{m\times p},
B\in \mathcal{F}^{m\times q},C\in \mathcal{F}^{s\times p},D\in \mathcal{F}^{p\times t},\Phi=\Phi^{*}\in \mathcal{F}^{m\times m},$ and $\Omega\in \mathcal{F}^{s\times t}$ be given. The system of matrix equations (\ref{system04}) has a pair of Hermitian solution $(X,Y)\in \mathcal{F}^{p\times p}\times \mathcal{F}^{q\times q}$ if and only if the ranks satisfy:
\begin{align*}
\mathbf{r}\begin{pmatrix}\Phi &A\\B^{*}&0\end{pmatrix}=\mathbf{r}(A)+\mathbf{r}(B),\mathbf{r}(A,~B,~\Phi)=\mathbf{r}(A,~B),
\end{align*}
\begin{align*}
\mathbf{r}(C,~\Omega)=\mathbf{r}(C),\mathbf{r}\begin{pmatrix}D\\ \Omega\end{pmatrix}=\mathbf{r}(D),
\mathbf{r}\begin{pmatrix}\Omega &0&C\\0&-\Omega^{*}&D^{*}\\D&C^{*}&0\end{pmatrix}=2\mathbf{r}\begin{pmatrix}C\\D^{*}\end{pmatrix},
\end{align*}
\begin{align*}
\mathbf{r}\begin{pmatrix}0&0&A^{*}&D\\A&B&-\Phi&0\\C&0&0&\Omega\end{pmatrix}=\mathbf{r}\begin{pmatrix}A&B\\C&0\end{pmatrix}+\mathbf{r}(A^{*},~D),
\end{align*}
\begin{align*}
\mathbf{r}\begin{pmatrix}0&A^{*}&D\\0&B^{*}&0\\A&-\Phi&0\\C&0&\Omega\end{pmatrix}=\mathbf{r}\begin{pmatrix}A^{*}&D\\B^{*}&0\end{pmatrix}+\mathbf{r}\begin{pmatrix}A\\C\end{pmatrix},
\end{align*}
\begin{align*}
\mathbf{r}\begin{pmatrix}0&0&A^{*}&D&C^{*}\\
0&0&B^{*}&0&0\\
A&A&-\Phi&0&0\\
C&0&0&\Omega&0\\
0&D&0&0&\Omega^{*}\end{pmatrix}=\mathbf{r}\begin{pmatrix}A&B\\C&0\\D^{*}&0\end{pmatrix}+\mathbf{r}\begin{pmatrix}A&A\\C&0\\0&D\end{pmatrix},
\end{align*}
\begin{align*}
\mathbf{r}\begin{pmatrix}0&0&0&A^{*}&D&C^{*}\\
A&A&B&-\Phi&0&0\\
C&0&0&0&\Omega&0\\
0&D^{*}&0&0&0&\Omega^{*}\end{pmatrix}=\mathbf{r}\begin{pmatrix}A\\C\\D^{*}\end{pmatrix}+\mathbf{r}\begin{pmatrix}A&A&B\\C&0&0\\0&D^{*}&0\end{pmatrix},
\end{align*}
\begin{align*}
\mathbf{r}\begin{pmatrix}0&0&0&0&0&0&B^{*}\\
0&0&0&C^{*}&0&A^{*}&0\\
0&0&0&0&D&A^{*}&A^{*}\\
0&C&0&0&-\Omega&0&0\\
0&0&D^{*}&\Omega^{*}&0&0&0\\
0&A&A&0&0&0&\Phi\\
B&0&A&0&0&-\Phi&0
\end{pmatrix}=2\mathbf{r}\begin{pmatrix}
0&A&B\\
A&A&0\\
C&0&0\\
0&D^{*}&0\end{pmatrix}.
\end{align*}

\end{theorem}

\begin{proof}The proof is similar to the proof of Theorem \ref{theorem05}.
\end{proof}

In the last part, we investigate another type of generalized Sylvester matrix equations (\ref{system05}).

\begin{theorem}\label{theorem6}
The system of matrix equations (\ref{system05}) is consistent if and only if the ranks satisfy:
\begin{align*}
\mathbf{r}\begin{pmatrix}\Phi&A&B\\H&0&0\end{pmatrix}=\mathbf{r}(A,~B)+\mathbf{r}(H),\mathbf{r}(C,~\Psi)=\mathbf{r}(C),\mathbf{r}(D,~\Omega)=\mathbf{r}(D),
\end{align*}
\begin{align*}
\mathbf{r}\begin{pmatrix}\Phi&B\\E&0\\H&0\end{pmatrix}=\mathbf{r}(B)+\mathbf{r}\begin{pmatrix}E\\H\end{pmatrix},
\mathbf{r}\begin{pmatrix}F\\ \Psi\end{pmatrix}=\mathbf{r}(F),\mathbf{r}\begin{pmatrix}G\\ \Omega\end{pmatrix}=\mathbf{r}(G),
\end{align*}
\begin{align*}
\mathbf{r}\begin{pmatrix}\Phi &0&A&B\\0&-\Psi&C&0\\E&F&0&0\\H&0&0&0\end{pmatrix}=\mathbf{r}\begin{pmatrix}A&B\\C&0\end{pmatrix}+\mathbf{r}\begin{pmatrix}E&F\\H&0\end{pmatrix},
\end{align*}
\begin{align*}
\mathbf{r}\begin{pmatrix}\Phi &0&A&B\\0&-\Omega&D&0\\E&G&0&0\\H&0&0&0\end{pmatrix}=\mathbf{r}\begin{pmatrix}A&B\\D&0\end{pmatrix}+\mathbf{r}\begin{pmatrix}E&G\\H&0\end{pmatrix},
\end{align*}
\begin{align*}
\mathbf{r}\begin{pmatrix}\Psi &0&C\\0&-\Omega&D\\F&G&0\end{pmatrix}=\mathbf{r}\begin{pmatrix}C\\D\end{pmatrix}+\mathbf{r}(F,~G),
\end{align*}
\begin{align*}
\mathbf{r}\begin{pmatrix}\Phi &0&0&A&A&B\\0&-\Psi&0&C&0&0\\0&0&-\Omega&0&D&0\\
E&F&G&0&0&0\\H&0&0&0&0&0\end{pmatrix}=\mathbf{r}\begin{pmatrix}A&A&B\\C&0&0\\0&D&0\end{pmatrix}+\mathbf{r}\begin{pmatrix}E&F&G\\H&0&0\end{pmatrix},
\end{align*}
\begin{align*}
\mathbf{r}\begin{pmatrix}\Phi &0&0&A&B\\0&-\Psi&0&C&0\\0&0&-\Omega&D&0\\E&F&0&0&0\\E&0&G&0&0\\H&0&0&0&0\end{pmatrix}=\mathbf{r}\begin{pmatrix}A&B\\C&0\\D&0\end{pmatrix}+
\mathbf{r}\begin{pmatrix}E&F&0\\E&0&G\\H&0&0\end{pmatrix}.
\end{align*}

\end{theorem}

\begin{proof}
It follows from Theorem \ref{theorem01} and Theorem \ref{theorem02} that the system of matrix equations (\ref{system05}) is equivalent to the system of matrix equations
\begin{align}
\left\{\begin{array}{rll}
S_{a}(P^{-1}XP_{1}^{-1})S_{e}+S_{b}(Q^{-1}YM_{1})+(MZQ_{1}^{-1})S_{h} & = &M \Phi M_{1},\\
S_{c}(P^{-1}XP_{1}^{-1})S_{f} & = & S \Psi S_{1},\\
S_{d}(P^{-1}XP_{1}^{-1})S_{g} & = &T \Omega T_{1}.
\end{array}
  \right.
\end{align}
Put
\begin{align}\label{equ312}
\widehat{X}=P^{-1}XP_{1}^{-1}:=\left(X_{ij}\right)_{18\times 18},~~\widehat{Y}:=Q^{-1}YM_{1}=\left(Y_{ij}\right)_{8\times 14},~~\widehat{Z}:=MZQ_{1}^{-1}=\left(Z_{ij}\right)_{14\times 8},
\end{align}
\begin{align}\label{equ313}
M \Phi M_{1}:=\left(\Phi_{ij}\right)_{14\times 14}, \ \ S \Psi S_{1}:=\left(\Psi_{ij}\right)_{10\times 10}, \ \ T\Omega T_{1}:=\left(\Omega_{ij}\right)_{10\times 10},
\end{align}where the block rows of $\left(\Phi_{ij}\right)_{14\times 14},\left(\Psi_{ij}\right)_{10\times 10},$ and $\left(\Omega_{ij}\right)_{10\times 10}$ are the same as the block rows of $S_{a},S_{c},$ and $S_{d}$, the block columns of $\left(\Phi_{ij}\right)_{14\times 14},\left(\Psi_{ij}\right)_{10\times 10},$ and $\left(\Omega_{ij}\right)_{10\times 10}$ are the same as the block columns of $S_{e},S_{f},$ and $S_{g}$. Using the similar approach in the proof of Theorem \ref{theorem04}, we can obtain the
solvability conditions to the system of matrix equations (\ref{system05}).
\end{proof}

As corollaries of Theorem \ref{theorem6}, we give the solvability conditions to the following systems over $\mathcal{F}$ with an involutive anti-automorphism ``$*$'' and the characteristic of $\mathcal{F}\neq2$:
\begin{align}\label{system06}
\left\{\begin{array}{c}
AXA^{*}+BY+(BY)^{*}=\Phi,\\CXC^{*}=\Psi,~DXD^{*}=\Omega,~X=X^{*}
\end{array}
  \right.
\end{align}
and
\begin{align}\label{system07}
\left\{\begin{array}{c}
AXA^{*}+BY+(BY)^{*}=\Phi,\\CXD=\Omega,~X=X^{*}
\end{array}
  \right.
\end{align}
where $X$ and $Y$ are unknowns and the others are matrices over $\mathcal{F}$ with compatible dimensions.

\begin{theorem}
Let $A\in \mathcal{F}^{m\times p},
B\in \mathcal{F}^{m\times q},C\in \mathcal{F}^{s\times p},D\in \mathcal{F}^{t\times p},\Phi=\Phi^{*}\in \mathcal{F}^{m\times m},\Psi=\Psi^{*}\in \mathcal{F}^{s\times s},$ and $\Omega=\Omega^{*}\in \mathcal{F}^{t\times t}$ be given. The system of matrix equations (\ref{system06}) is consistent if and only if the ranks satisfy:
\begin{align*}
\mathbf{r}\begin{pmatrix}\Phi &A&B\\B^{*}&0&0\end{pmatrix}=\mathbf{r}(A,~B)+\mathbf{r}(B),~~\mathbf{r}(C,~\Psi)=\mathbf{r}(C),~~\mathbf{r}(D,~\Omega)=\mathbf{r}(D),
\end{align*}
\begin{align*}
\mathbf{r}\begin{pmatrix}\Phi &0&A&B\\0&-\Psi&C&0\\A^{*}&C^{*}&0&0\\B^{*}&0&0&0\end{pmatrix}=2\mathbf{r}\begin{pmatrix}A&B\\C&0\end{pmatrix},
\mathbf{r}\begin{pmatrix}\Phi&0&A&B\\0&-\Omega&D&0\\A^{*}&D^{*}&0&0\\B^{*}&0&0&0\end{pmatrix}=2\mathbf{r}\begin{pmatrix}A&B\\D&0\end{pmatrix},
\end{align*}
\begin{align*}
\mathbf{r}\begin{pmatrix}\Psi&0&C\\0&-\Omega&D\\C^{*}&D^{*}&0\end{pmatrix}=2\mathbf{r}\begin{pmatrix}C\\D\end{pmatrix},
\end{align*}
\begin{align*}
\mathbf{r}\begin{pmatrix}\Phi&0&0&A&A&B\\0&-\Psi&0&C&0&0\\0&0&-\Omega&0&D&0\\A^{*}&C^{*}&D^{*}&0&0&0\\B^{*}&0&0&0&0&0\end{pmatrix}
=\mathbf{r}\begin{pmatrix}A&A&B\\C&0&0\\0&D&0\end{pmatrix}+\mathbf{r}\begin{pmatrix}A&B\\C&0\\D&0\end{pmatrix}.
\end{align*}

\end{theorem}

\begin{proof}

We first prove that the system of matrix equations (\ref{system06}) is consistent
if and only if the system of matrix equations
\begin{align} \label{equ345}
\left\{\begin{array}{rll}
A\widetilde{X}A^{*}+B\widetilde{Y}+\widetilde{Z}B^{*} & = &\Phi,\\
C\widetilde{X}C^{*} & = &\Psi,\\
D\widetilde{X}D^{*}& = & \Omega
\end{array}
  \right.
\end{align}
has a solution. If the system (\ref{system06}) has a solution,
say, $(X_{0},Y_{0})$, then the system (\ref{equ345}) clearly
has a solution $(\widetilde{X},\widetilde{Y},\widetilde{Z})=(X_{0},Y_{0},Y_{0}^{*})$. Conversely, if the system (\ref{equ345}) has a solution
$(\widetilde{X},\widetilde{Y},\widetilde{Z})$, then
\begin{align*}
(X,Y)=\Big(\frac{\widetilde{X}+\widetilde{X}^{*}}{2},\frac{\widetilde{Y}+\widetilde{Z}^{*}}{2}\Big)
\end{align*}
is a solution of (\ref{system06}). We can derive the
solvability conditions to the system of matrix equations (\ref{system06}) by
Theorem \ref{theorem6}.

\end{proof}

Similarly, we can give the solvability conditions to the system of matrix equations (\ref{system07}) as follows.

\begin{theorem}
Let $A\in \mathcal{F}^{m\times p},
B\in \mathcal{F}^{m\times q},C\in \mathcal{F}^{s\times p},D\in \mathcal{F}^{p\times t},\Phi=\Phi^{*}\in \mathcal{F}^{m\times m},$ and $\Omega\in \mathcal{F}^{s\times t}$ be given. The system of matrix equations (\ref{system07}) is consistent if and only if the ranks satisfy:
\begin{align*}
\mathbf{r}\begin{pmatrix}\Phi&A&B\\B^{*}&0&0\end{pmatrix}=\mathbf{r}(A,~B)+\mathbf{r}(B),\mathbf{r}(C,~\Omega)=\mathbf{r}(C),\mathbf{r}\begin{pmatrix}D\\ \Omega\end{pmatrix}=\mathbf{r}(D),
\end{align*}
\begin{align*}
\mathbf{r}\begin{pmatrix}\Phi &0&A&B\\0&-\Omega&C&0\\A^{*}&D&0&0\\B^{*}&0&0&0\end{pmatrix}=\mathbf{r}\begin{pmatrix}A&B\\C&0\end{pmatrix}+\mathbf{r}\begin{pmatrix}A&B\\D^{*}&0\end{pmatrix},
\end{align*}
\begin{align*}
\mathbf{r}\begin{pmatrix}\Omega&0&C\\0&-\Omega^{*}&D^{*}\\D&C^{*}&0\end{pmatrix}=2\mathbf{r}\begin{pmatrix}C\\D^{*}\end{pmatrix},
\end{align*}
\begin{align*}
\mathbf{r}\begin{pmatrix}\Phi&0&0&A&A&B\\0&-\Omega&0&C&0&0\\0&0&-\Omega^{*}&0&D^{*}&0\\A^{*}&D&C^{*}&0&0&0\\B^{*}&0&0&0&0&0\end{pmatrix}
=\mathbf{r}\begin{pmatrix}A&A&B\\C&0&0\\0&D^{*}&0\end{pmatrix}+\mathbf{r}\begin{pmatrix}A&B\\C&0\\D^{*}&0\end{pmatrix}.
\end{align*}

\end{theorem}

\section{\textbf{Conclusion}}

We have presented all the dimensions of identity matrices in the equivalence canonical form of four matrices over an arbitrary division ring $\mathcal{F}$
with compatible sizes: $A\in \mathcal{F}^{m\times p},
B\in \mathcal{F}^{m\times q},C\in \mathcal{F}^{s\times p},D\in \mathcal{F}^{t\times p}$, which is a challenging problem proposed in \cite{zhangxia}.  Using the complete equivalence canonical form, we have derived some necessary and sufficient conditions for the existence of the general solutions to the systems (\ref{system02}), (\ref{system05}), and (\ref{system03})-(\ref{system07}) in terms of  ranks of the given matrices.

Note that the real number field, the complex number field, and the real quaternion algebra are special cases of an arbitrary division ring $\mathcal{F}$. Hence, all the results in this paper are also valid over the real number field, the complex number field and the real quaternion algebra.
\\

\noindent {\bf Acknowledgement:}  The authors would like to express their sincere thanks to   the referee for his/her
constructive and valuable comments and suggestions which greatly
improve this paper.

\end{document}